\documentclass[12pt,a4paper]{amsart}

\usepackage[utf8]{inputenc}
\usepackage[T1]{fontenc}
\usepackage[american]{babel}
\usepackage{lmodern}
\usepackage{amsmath,amssymb,amsthm}
\usepackage{geometry}
\usepackage{hyperref}

\renewcommand{\a}{\mathfrak{a}}
\renewcommand{\b}{\mathfrak{b}}
\newcommand{\C}{\mathcal{C}}
\newcommand{\E}{\mathcal{E}}
\newcommand{\IR}{\mathbb{R}}

\newcommand{\Ch}{\mathrm{Ch}}
\newcommand{\Lipbs}{\mathrm{Lip}_{\mathrm{bs}}}
\renewcommand{\Re}{\operatorname{Re}}
\newcommand{\sgn}{\operatorname{sgn}}
\newcommand{\supp}{\operatorname{supp}}
\newcommand{\vol}{\mathrm{vol}}

\newcommand{\lra}{\longrightarrow}

\newcommand{\abs}[1]{\lvert#1\rvert}
\newcommand{\norm}[1]{\lVert#1\rVert}

\makeatletter\def\blfootnote{\xdef\@thefnmark{}\@footnotetext}\makeatother

\newtheorem{theorem}{Theorem}[section]
\newtheorem{corollary}[theorem]{Corollary}
\newtheorem{lemma}[theorem]{Lemma}
\newtheorem{definition}[theorem]{Definition}

\theoremstyle{remark}

\title{Stability of Kac regularity under domination of quadratic forms}
\author[Wirth]{Melchior Wirth}
\address{M. Wirth, Mathematisches Institut, Friedrich-Schiller-Universität Jena, D-07737 Jena, Germany}
\email{melchior.wirth@uni-jena.de}

\begin{document}
\maketitle


\begin{abstract}
A domain is called Kac regular for a quadratic form on $L^2$ if the closure of all functions vanishing almost everywhere outside a closed subset of the domain coincides with the set of all functions vanishing almost everywhere outside the domain. It is shown that this notion is stable under domination of quadratic forms. As applications measure perturbations of quasi-regular Dirichlet forms, Cheeger energies on metric measure spaces and Schrödinger operators on manifolds are studied. Along the way a characterization of the Sobolev space with Dirichlet boundary conditions on domains in infinitesimally Riemannian metric measure spaces is obtained.
\end{abstract}

\tableofcontents

\section*{Introduction}

Following Stroock, a domain $\Omega\subset\IR^n$ is called Kac regular if the first exit time of the Brownian motion from $\Omega$ equals the first penetration time of the Brownian motion to $\Omega^c$. In analytic terms, Kac regularity has been proven by Herbst and Zhao \cite{HZ88} to be equivalent to the property that every $u\in W^{1,2}(\IR^n)$ with $u=0$ a.e. on $\Omega^c$ can be approximated in $W^{1,2}$ by elements of $C_c^\infty(\Omega)$.

In a recent preprint by Bei and Güneysu \cite{BG17} it has been shown that Kac regularity enjoys a remarkable stability in the following sense: If $\Omega$ is Kac regular, then the characterizing approximation property holds not only for $W^{1,2}$, but also for the form domain of a large class of Schrödinger operators. It is the aim of this short note to show how this stability can abstractly be understood in terms of domination of quadratic forms.

In \cite{BG17}, the results are derived using deep techniques from stochastic analysis, in particular Feynman-Kac formulae and stochastic parallel transport. In contrast, our approach is purely functional analytical, drawing on the result of a joint article with Lenz and Schmidt \cite{LSW16}, and works in the setting of (quasi-regular) Dirichlet forms so that it can readily be applied not only to quadratic forms on Riemannian manifolds, but also on (infinitesimally Riemannian) metric measure spaces. In this context we also give a characterization of the Sobolev space with Dirichlet boundary conditions on domains in infinitesimally Riemannian metric measure spaces, which might be of independent interest.

This article is structured as follows: In Section \ref{abstract_theory} we introduce the notion of Kac regularity with respect to quadratic forms on $L^2$-spaces and prove the abstract stability result under domination of quadratic forms (Theorem \ref{Kac_regular_stable}). In Section \ref{Dirichlet_forms} we study Kac regular domains for quasi-regular Dirichlet forms, collect several equivalent definitions of the form domain with Dirichlet boundary conditions and prove that Theorem \ref{Kac_regular_stable} implies the stability of Kac regularity under measure perturbations (Theorem \ref{stable_meas_pert}). In Section \ref{mms} we apply the results of Section \ref{Dirichlet_forms} to the Cheeger energy on infinitesimally Riemannian metric measure spaces (Theorems \ref{char_Sobolev}, \ref{stability_potential}). Finally, in Section \ref{Euclidean} we show how the stability result of \cite{BG17} fits into our framework (Theorem \ref{stability_Schrödinger}).

\textbf{Acknowledgment:} The author would like to thank Daniel Lenz for his support and encouragement during the author's ongoing graduate studies and him as well as Marcel Schmidt for fruitful discussions on domination of quadratic forms. He would also like to thank the organizers of the conference \emph{Analysis and Geometry on Graphs and Manifolds} in Potsdam, where the initial motivation of this article was conceived, and the organizers of the intense activity period \emph{Metric Measure Spaces and Ricci Curvature} at MPIM in Bonn, where this work was finished. Financial support by the German Academic Scholarship Foundation (Studienstiftung des deutschen Volkes) and by the German Research Foundation (DFG) via RTG 1523/2 is gratefully acknowledged.

\section{Kac regular domains for quadratic forms}\label{abstract_theory}
In this section we introduce the concept of Kac regular domains for quadratic forms and prove an abstract stability theorem under domination.

Let $X$ be a topological space, $m$ a Borel measure on $X$ and $E\lra X$ a Hermitian vector bundle. If $\a$ is a closed quadratic form on $L^2(X;E)$, let 
\begin{align*}
\norm\cdot_\a=(\norm\cdot_{L^2}^2+\a(\cdot))^{\frac 1 2}.
\end{align*}
For an open subset $\Omega$ of $X$ denote by $D(\a_\Omega)$ the $\norm\cdot_\a$-closure of the set of all $\Phi\in D(\a)$ with $\supp\Phi\subset \Omega$. Here $\supp\Phi$ is understood as the support of the measure $\supp\abs{\Phi}m$. Further let $\a_\Omega$ be the restriction of $\a$ to $D(\a_\Omega)$.

Moreover let $D(\tilde\a_\Omega)=\{\Phi\in D(\a)\mid \Phi=0\text{ a.e. on }\Omega^c\}$ and denote by $\tilde\a_\Omega$ the restriction of $\a$ to $D(\tilde \a_\Omega)$.

\begin{definition}
We call $\Omega$ \emph{Kac regular} for $\a$ if $D(\a_\Omega)=D(\tilde \a_\Omega)$.
\end{definition}

\begin{lemma}
The space $D(\tilde\a_\Omega)$ is closed in $D(\a)$.
\end{lemma}
\begin{proof}
This follows immediately from the fact that every $\norm\cdot_\a$-convergent sequences also converges in $L^2$ and hence has an a.e. convergent subsequence.
\end{proof}

Now let us turn to domination of quadratic forms (see \cite{Ouh96,MVV05}). The closed quadratic form $\a$ on $L^2(X,m;E)$ is said to be dominated by the closed quadratic form $\b$ on $L^2(X,m)$ if for all $\Phi\in D(\a)$ and $v\in D(\b)$ with $0\leq v\leq \abs{\Phi}$ one has $\abs{\Phi}\in D(\b)$, $v\sgn \Phi\in D(\a)$ and
\begin{align*}
\b(v,\abs{\Phi})\leq \Re \a(\Phi,v \sgn \Phi).
\end{align*}
Here $\sgn\Phi(x)=\Phi(x)/\abs{\Phi(x)}_x$ whenever $\Phi(x)\neq 0$ (the value in the case $\Phi(x)=0$ is obviously irrelevant for the preceding definition).

We will use the following two facts (see \cite{Ouh96}, Corollary 2.5 and Proposition 3.2): The dominating form $\b$ necessarily satisfies the first Beurling-Deny criterion, i.e. $\b(\abs{u})\leq \b(u)$ for all $u\in L^2(X,m)$, and conversely, every form satisfying the first Beurling-Deny criterion is dominated by itself.

\begin{lemma}\label{domination_domains}
Let $\a$ be a closed quadratic form on $L^2(X,m;E)$, $\b$ a closed quadratic form on $L^2(X,m)$ and $\Omega\subset X$ open. If $\a$ is dominated by $\b$, then $\a_\Omega$ is dominated by $\b_\Omega$ and $\tilde \a_\Omega$ is dominated by $\tilde\b_\Omega$.
\end{lemma}
\begin{proof}
If $\Phi\in D(\tilde b_\Omega)$ and $v\in D(\tilde a_\Omega)$ with $0\leq v\leq\abs{\Phi}$, then $v\sgn \Phi\in D(\a)$ and $\abs{\Phi}\in D(\b)$ since $\a$ is dominated by $\b$. Moreover, $\abs{\Phi},v\sgn \Phi=0$ a.e. on $\Omega^c$. Thus $\abs{\Phi}\in D(\tilde\b_\Omega)$ and $v\sgn\Phi\in D(\tilde \a_\Omega)$. The inequality between $\tilde \a_\Omega$ and $\tilde \b_\Omega$ follows directly by restriction. The proof for $\a_\Omega$, $\b_\Omega$ is analogous.
\end{proof}

\begin{theorem}\label{Kac_regular_stable}
Let $\a$ be a closed quadratic form on $L^2(X,m;E)$, $\b$ a closed quadratic form on $L^2(X,m)$ and $\Omega\subset X$ open. If $\a$ is dominated by $\b$ and $\Omega$ is Kac regular for $\b$, then it is Kac regular for $\a$.
\end{theorem}
\begin{proof}
By Lemma \ref{domination_domains}, the form $\tilde\a_\Omega$ is dominated by $\tilde \b_\Omega$. Let $D_\a$ and $D_\b$ be the set of all elements of $D(\a)$ and $D(\b)$ respectively with support in $\Omega$. By assumption, $D_\b$ is dense in $D(\tilde\b_\Omega)$. It follows from \cite{LSW16}, Theorem 2.2 (see also Corollary 2.3) that $D_\a$ is dense in $D(\tilde\a_\Omega)$. Thus $\Omega$ is Kac regular for $\a.$
\end{proof}

\section{Quasi-regular Dirichlet forms}\label{Dirichlet_forms}
In this section we give a characterization of $D(\E_\Omega)$ that is better suited for applications in the case when $\E$ is a quasi-regular Dirichlet form.

The definition of quasi-regular Dirichlet forms along with all necessary properties can be found in \cite{MR92}. At this point let us just mention that every Dirichlet form satisfies the first Beurling-Deny criterion and is thus dominated by itself.

Further let us recall some standard terminology: An ascending sequence $(F_k)$ of closed subsets of $X$ is called \emph{nest} if $\bigcup_k D(\E)_{F_k}$ is dense in $\E$. A Borel set $B\subset X$ is called \emph{exceptional} if $B\subset \bigcap_k F_k^c$ for some nest $(F_k)$. A property is said to hold \emph{quasi-everywhere}, abbreviated \emph{q.e.}, if it holds outside an exceptional set. A function $u$ is called \emph{quasi-continuous} if there exists a nest $(F_k)$ such that $u|_{F_k}$ is continuous for all $k\in\mathbb{N}$. Every element $u$ of the domain of a quasi-regular Dirichlet form has a quasi-continuous version, determined up to equality q.\,e. We will denote it by $\tilde u$.

Let $X$ be a Hausdorff space, $m$ a $\sigma$-finite Borel measure on $X$ of full support, and $\E$ a quasi-regular Dirichlet form on $L^2(X,m)$. For a Borel set $B\subset X$ let
\begin{align*}
D(\E)_B=\{u\in D(\E)\mid u=0\text{ q.e. on }B^c\}.
\end{align*}

Alternatively, $D(\E)_B$ can be characterized as the closure of all functions with compact support in $B$, as the following lemma (\cite{RS95}, Lemma 2.7) shows.
\begin{lemma}\label{compactly_supported}
For every Borel set $B\subset X$ there exists an ascending sequence $(F_k)$ of compact subsets of $B$ such that $\bigcup_k D(\E)_{F_k}$ is $\norm\cdot_\E$-dense in $D(\E)_B$.
\end{lemma}
\begin{corollary}\label{char_q.e.}
If $\Omega\subset X$ is open, then $D(\E_\Omega)=D(\E)_\Omega$.
\end{corollary}

Note that in this situation, $\E_\Omega$ is again a quasi-regular Dirichlet form (\cite{RS95}, Lemma 2.12).

In many cases it is more customary to define the Sobolev space with Dirichlet boundary conditions not as the closure of all compactly supported functions, but as the closure of a certain subset (e.g. continuous functions or smooth ones). The following definition gives a general setup for these situations.

\begin{definition}
A subalgebra $\C$ of $D(\E)\cap C_b(X)$ is called \emph{generalized special standard core} if 
\begin{itemize}
\item $\C$ is dense in $D(\E)$,
\item for every $\epsilon>0$ there exists an increasing $1$-Lipschitz function $C_\epsilon\colon \IR\lra\IR$ such that $C_\epsilon(t)=t$ for $t\in [0,1]$, $-\epsilon\leq C_\epsilon\leq 1+\epsilon$ and $C_\epsilon\circ u\in \C$ for all $u\in\C$,
\item For every compact $K\subset X$ and every open $G\subset X$ with $K\subset G$ there exists a $u\in\C$ such that $u\geq 0$, $u|_K=1$ and $\supp u\subset G$.
\end{itemize}
\end{definition}
As the named suggest, this concept generalizes special standard cores in the regular setting (see \cite{FOT11}, Section 1.1). In particular, if $\E$ is a regular Dirichlet form, then $D(\E)\cap C_c(X)$ is a special standard core. More examples will be discussed in the following sections.

Given the characterization from Lemma \ref{compactly_supported}, the following lemma can be proven exactly as Lemma 2.3.4 in \cite{FOT11}.
\begin{lemma}\label{char_core}
If $\C$ is a generalized special standard core, then $\{u\in \C\mid \supp u\subset \Omega\}$ is $\norm\cdot_\E$-dense in $D(\E_\Omega)$.
\end{lemma}

Next we study a class of perturbations of quasi-regular Dirichlet forms that are dominated by the original form.

A Borel measure $\mu$ on $X$ is called \emph{smooth} if $\mu(B)=0$ for every exceptional set $B$ and there exists a nest $(F_k)$ of compact sets such that $\mu(F_k)<\infty$ for all $k\in\mathbb{N}$.

If $\mu$ is a smooth measure, define the quadratic form $\E^\mu$ by
\begin{align*}
D(\E^\mu)=\{u\in D(\E)\mid \tilde u\in L^2(X,\mu)\},\,\E^\mu(u)=\E(u)+\int_X\tilde u^2\,d\mu.
\end{align*}
The form $\E^\mu$ is again a quasi-regular Dirichlet form (\cite{RS95}, Proposition 2.3).

\begin{lemma}\label{measure_pert_dom}
The form $\E^\mu$ is dominated by $\E$.
\end{lemma}
\begin{proof}
Let $u\in D(\E^\mu)$ and $v\in D(\E)$ with $0\leq v\leq\abs{u}$ a.e. Then $\abs{u}\in D(\E)$ and $v\sgn u\in D(\E)$ since $\E$ is dominated by itself. Furthermore, 
\begin{align*}
\int_X\abs{\widetilde{v\sgn u}}^2\,d\mu=\int_X\abs{\tilde v}^2\,d\mu\leq \int_X\abs{\tilde u}^2\,d\mu<\infty.
\end{align*}
Thus $v\sgn u\in D(\E^\mu)$.

Finally,
\begin{align*}
\E^\mu(u,v\sgn u)=\E(u,v\sgn u)+\int \tilde v\abs{\tilde u}\,d\mu\geq \E(u,v\sgn u)\geq \E(\abs{u},v)
\end{align*}
since $\E$ is dominated by itself.
\end{proof}

Now we can combine Lemma \ref{measure_pert_dom} with Theorem \ref{Kac_regular_stable} to obtain stability of Kac regularity under measure perturbations.
\begin{theorem}\label{stable_meas_pert}
If an open set $\Omega\subset X$ is Kac regular for $\E$, then it is Kac regular for $\E^\mu$ for every smooth measure $\mu$.
\end{theorem}

\section{Metric measure spaces}\label{mms}

In this section we apply the results of the previous section to the Cheeger energy on infinitesimally Riemannian metric measure spaces.

Let $(X,d)$ be a complete, separable metric space and $m$ a $\sigma$-finite Borel measure of full support on $X$ satisfying $m(B_r(x))<\infty$ for all $x\in X$, $r>0$. Denote by $\mathrm{Lip}_{\mathrm{b}}(X,d)$ the space of all Lipschitz functions on $X$. For $f\in \mathrm{Lip}_{\mathrm{b}}(X,d)$ the local Lipschitz constant is defined as
\begin{align*}
\abs{Df}(x)=\limsup_{y\to x}\frac{\abs{f(x)-f(y)}}{d(x,y)},
\end{align*}
The Cheeger energy $\Ch$ is the $L^2$-lower semicontinuous relaxation of
\begin{align*}
\mathrm{Lip}_{\mathrm{b}}(X,d)\cap L^2(X,m)\lra [0,\infty],\,f\mapsto \frac 1 2\int_X \abs{Df}^2\,dm.
\end{align*}

If $\Ch$ is a quadratic form, then $(X,d,m)$ is called \emph{infinitesimally Riemannian} (see e.g. \cite{AGS15,Gig15}). In this case, $\Ch$ is a quasi-regular Dirichlet form (\cite{Sav14}, Theorem 4.1 -- the theorem is formulated for $\mathrm{RCD}$ spaces, but the proof of quasi-regularity holds under the weaker assumptions stated above).

Let $(X,d,m)$ be an infinitesimally Riemannian metric measure space and $\Ch$ the Cheeger energy. The space $\Lipbs(X,d)$ of Lipschitz functions with bounded support is $\norm\cdot_{\Ch}$-dense in $D(\Ch)$, as follows e.g. from \cite{AGS14}, Lemma 4.3, combined with a standard localization argument. From that fact it is easy to see that $\Lipbs(X,d)$ is indeed a generalized special standard core for $\Ch$.

The results from the last section (Corollary \ref{char_q.e.}, Lemma \ref{char_core}) immediately give the following characterization of the first order Sobolev space with Dirichlet boundary conditions on domains in infinitesimally Riemannian metric measure spaces.

\begin{theorem}\label{char_Sobolev}
For $\Omega\subset X$ open, the following spaces all coincide:
\begin{itemize}
\item The $\norm\cdot_{\Ch}$-closure of $\{u\in D(\Ch)\mid \supp u\subset \Omega,\,\supp u\text{ compact}\}$. 
\item The $\norm\cdot_{\Ch}$-closure of $\{u\in\Lipbs(X,d)\mid \supp u\subset \Omega\}$.
\item The set $\{u\in D(\Ch)\mid \tilde u=0\text{ q.e. on }\Omega^c\}$.
\end{itemize}
\end{theorem}

Now we turn to perturbations of the Cheeger energy by a positive potential.

For measurable $V\colon X\lra [0,\infty]$ let
\begin{align*}
\Ch^V\colon L^2(X,m)\lra [0,\infty],\,\Ch_V(u)=\Ch(u)+\int_X \abs{u}^2V\,dm.
\end{align*}

\begin{theorem}\label{stability_potential}
If $\int_K V\,dm<\infty$ for all compact $K\subset X$, then $\Ch^V$ is a quasi-regular Dirichlet form, and if $\Omega$ is Kac regular for $\Ch$, so it is for $\Ch^V$.
\end{theorem}
\begin{proof}
Since $\Ch$ is quasi-regular, there is a nest $(F_k)_k$ of compact sets. By assumption $Vm(F_k)<\infty$ for all $k\in\mathbb{N}$. Of course, $Vm$ does not charge measures of capacity zero. Hence $Vm$ is a smooth measure, and the assertions follow from Theorem \ref{char_core}.
\end{proof}

\section{Schrödinger operators on Riemannian manifold}\label{Euclidean}

In this section Kac regularity for quadratic forms generated by Schrödinger operators on vector bundles over manifolds is examined.

Let $(M,g)$ be a Riemannian manifold and
\begin{align*}
\E\colon W^{1,2}_0(M)\lra [0,\infty],\,u\mapsto \int_M \abs{\nabla u}^2\,d\vol_g.
\end{align*}
The space $C_c^\infty(\Omega)$ is obviously a special standard core for $\E$ and so $D(\E_\Omega)$ coincides with $W^{1,2}_0(\Omega)$ for $\Omega\subset M$ open.

Further let $E\lra M$ be a Hermitian vector bundle with metric covariant derivative $\nabla$ and let $V\colon M\lra \mathrm{End}(E)$ be a measurable section with $V(x)\geq 0$ for a.e. $x\in M$. Denote by $\Gamma_{C_c^\infty}(M;E)$ the smooth section of $E$ with compact support and let $\E^{\nabla,V}$ be the closure of
\begin{align*}
\Gamma_{C_c^\infty}(M;E)\lra [0,\infty),\,\Phi\mapsto\int_M (\abs{\nabla \Phi(x)}_x^2+\langle V(x)\Phi(x),\Phi(x)\rangle_x)\,d\vol_g(x).
\end{align*}
From the proof of \cite{Gue14}, Proposition 2.2, together with the semigroup characterization of domination (\cite{LSW16}, Theorem 1.33) it follows that $\E^{\nabla,0}$ is dominated by $\E$. Given this result it is not hard to see (with an argument along the lines of Lemma \ref{measure_pert_dom}) that $\E^{\nabla,V}$ is dominated by $\E$ as well. Thus we can apply Theorem \ref{Kac_regular_stable} to obtain the stability of Kac regularity in this setting.

\begin{theorem}\label{stability_Schrödinger}
If $\Omega$ is Kac regular for $\E$, then it is also Kac regular for $\E^{\nabla,V}$.
\end{theorem}
Note that $D(\E^{\nabla,V}_\Omega)$ is a priori larger than the closure of $\Gamma_{C_c^\infty}(\Omega;E)$. However, it follows from a standard approximation argument that these spaces indeed coincide and so the previous theorem recovers  Theorem 2.13 a) of \cite{BG17} in the case of \emph{nonnegative} potentials.

\end{document}